\documentclass{amsart}

\usepackage{amsmath,amssymb,amsfonts,amstext,amsthm}
\usepackage{url}
\usepackage{graphicx}
\input xy
\xyoption{all}
\usepackage[all]{xy}
\usepackage{mathrsfs}

\graphicspath{{images/}{../images/}}
\usepackage{subfiles}
\usepackage{tikz}

\newcommand{\nwc}{\newcommand}

\nwc{\aaa}{\mathcal{F}}
\nwc{\aap}{\mathcal{F}_{P}}
\nwc{\al}{\alpha}

\nwc{\C}{\mathbb{C}}
\nwc{\cb}{\overline{C}}
\nwc{\ccc}{\mathfrak{c}}
\nwc{\ch}{\widehat{C}}
\nwc{\cin}{\textbf{(v)}}
\nwc{\cl}{C'}
\nwc{\cp}{\mathcal{C}_{P}}
\nwc{\cpll}{\mathfrak{c}_{P'}}
\nwc{\ct}{\widetilde{C}}

\nwc{\dd}{\mathcal{L}}
\nwc{\ddd}{\mathfrak{d}}
\nwc{\ddl}{\mathcal{L}'}
\nwc{\dlp}{\delta_{P}}
\nwc{\doi}{\textbf{(ii)}}

\nwc{\enq}{$$}

\nwc{\fl}{\flushleft}
\nwc{\fff}{\mathcal{F}}
\nwc{\ffp}{\mathcal{F}_{P}}
\nwc{\ffq}{\mathcal{F}_{Q}}
\nwc{\ffl}{\mathcal{F}'}

\nwc{\G}{\mathcal{G}}
\nwc{\Ga}{\Gamma}
\nwc{\gtl}{\widetilde{g}}

\nwc{\hra}{\hookrightarrow}
\nwc{\hua}{h^{1}(C,\aaa )}

\nwc{\kk}{{\rm K}}

\nwc{\llb}{\mathcal{L}}

\nwc{\mb}{\mathbb}
\nwc{\mc}{\mathcal}
\nwc{\mm}{\mathfrak{m}}
\nwc{\mmp}{\mathfrak{m}_{P}}
\nwc{\mpd}{\mathfrak{m}_{P}^{2}}

\nwc{\nn}{\mathbb{N}}

\nwc{\ob}{\overline{\mathcal{O}}}
\nwc{\obr}{\mathcal{O}^*}
\nwc{\obp}{\overline{\mathcal{O}}_P}
\nwc{\och}{\mathcal{O}_{\hat{C}}}
\nwc{\oh}{\hat{\mathcal{O}}}
\nwc{\ohp}{\hat{\mathcal{O}}_{P}}
\nwc{\ol}{\mathcal{O}'}
\nwc{\oma}{\Omega (\mathfrak{a})}
\nwc{\omo}{\Omega (\mathcal{O})}
\nwc{\oo}{\mathcal{O}}
\nwc{\op}{\mathcal{O}_P}
\nwc{\opc}{\mathcal{O}_{P,C}}
\nwc{\oph}{\hat{\mathcal{O}}_{P}}
\nwc{\opl}{\mathcal{O}_{P}'}
\nwc{\oplc}{\mathcal{O}_{P,C}'}
\nwc{\opll}{\mathcal{O}_{P'}}
\nwc{\opt}{\tilde{\mathcal{O}}_{P}}
\nwc{\optt}{{\mathcal{O}}_{\tilde{P}}}
\nwc{\oq}{\mathcal{O}_{Q}}
\nwc{\oqt}{\tilde{\mathcal{O}}_{Q}}
\nwc{\ot}{\widetilde{\mathcal{O}}}
\nwc{\overop}{\bar{\oo}_{P}}

\nwc{\pb}{\overline{P}}
\nwc{\pbb}{P^*}
\nwc{\pbi}{\overline{P_{i}}}
\nwc{\pbr}{\overline{P_{r}}}
\nwc{\pgmd}{\mathbb{P}^{g+2}}
\nwc{\pgmu}{\mathbb{P}^{g+1}}
\nwc{\ph}{\hat{P}}
\nwc{\pp}{\mathbb{P}}
\nwc{\prv}{\noindent\textbook{Proof}:}
\nwc{\pt}{\widetilde{P}}
\nwc{\ptl}{\tilde{P}}
\nwc{\pum}{\mathbb{P}^{1}}

\nwc{\qh}{\hat{Q}}
\nwc{\qtl}{\tilde{Q}}
\nwc{\qua}{\textbf{(iv)}}

\nwc{\ra}{\rightarrow}
\nwc{\rh}{\hat{R}}

\nwc{\sei}{\textbf{(vi)}}
\nwc{\sep}{\beq\ast\ \ast\ \ast\enq}
\nwc{\sig}{\sigma}
\nwc{\Sig}{\Sigma}
\nwc{\ssp}{S_{P}}
\nwc{\sss}{{\rm S}}

\nwc{\tre}{\textbf{(iii)}}

\nwc{\um}{\textbf{(i)}}

\nwc{\vpb}{v_{\overline{P}}}
\nwc{\vtxp}{\widetilde{V}_{x,P}}
\nwc{\vxp}{V_{x,P}}

\let \mc=\mathcal

\nwc{\wh}{\hat{\omega}}
\nwc{\whp}{\hat{\omega}_{P}}
\nwc{\woch}{\omega\cdot\mathcal{O}_{\hat{C}}}
\nwc{\woh}{\omega\cdot\hat{\mathcal{O}}}
\nwc{\ww}{\omega}
\nwc{\wwb}{\omega^*}
\nwc{\wwct}{\omega _{\widetilde{C}}}
\nwc{\wwh}{\widehat{\omega}}
\nwc{\wwhp}{\widehat{\omega}_P}
\nwc{\wwp}{\omega _{P}}
\nwc{\wwt}{\widetilde{\omega}}
\nwc{\wwtp}{\widetilde{\omega}_P}

\nwc{\zz}{\mathbb{Z}}

\newtheorem{coro}{Corollary}[section]

\newtheorem{lemma}[coro]{Lemma}

\newtheorem{conj}[coro]{Conjecture}

\let \fl=\flushleft

\let \ga=\gamma

\let \al=\alpha

\let \la=\lambda

\usepackage[margin=3cm]{geometry}
\usepackage{xy}
\usepackage{graphics}
\title{Inflection divisors of linear series on an elliptic curve}
\author{Ethan Cotterill}

\address{Instituto de Matem\'atica, UFF, Rua Prof Waldemar de Freitas, S/N,
24.210-201 Niter\'oi RJ, Brazil}

\email{cotterill.ethan@gmail.com}

\author{Cristhian Garay L\'opez}

\address{Departamento de Matem\'aticas, Centro de Investigaci\'on y de Estudios Avanzados del IPN, Apartado Postal 14-740, 07000. Ciudad de M\'exico, M\'exico.}

\email{cgaray@math.cinvestav.mx}

\begin{document}
\maketitle{}

\begin{abstract}
In this largely-expository note, we describe a class of divisors on elliptic curves that index the inflection points of linear series arising (as subspaces of holomorphic sections) from line bundles on $\mb{P}^1$ via pullback along the canonical 2-to-1 projection. Associated to each inflection divisor on an elliptic curve $E_{\la}: y^2= x(x-1)(x-\la)$, there is an associated {\it inflectionary curve} in (the projective compactification of) the affine plane in coordinates $x$ and $\la$. These inflectionary curves have remarkable features; among other things, they lead directly to an explicit conjecture for the number of {\it real} inflection points of linear series on $E_{\la}$ whenever the Legendre parameter $\la$ is real.
\end{abstract}

%\vspace{-20pt}
\section{Introduction}
% \textbf{The ingredients.} Let $K$ be a field of characteristic zero, $E=(E,O)$ an elliptic curve over $K$ and $k,g$ two  integers  satisfying $k>g>0$. We set  $n:=k-g$, so that $1\leq n\leq k-1$.

\subsection{Motivation} The {\it torsion points} of an elliptic curve $E$ defined over a field $K$ are classical objects of study, and they are parametrized by the {\it division polynomials} defined in \cite{S}. Given integers $k>g>0$, set $\mu:=k-g$. Below we introduce an effective divisor $I(\mu,k)$ on $E$ of degree $4\mu(k+1)$ that generalizes the divisor $E[2k]$ of order-$2k$ torsion points, in that
\[
I(k-1,k)=\sum_{p\in E[2k]}p-R_\pi
\]
where $R_\pi$ is the ramification divisor of the double cover $\pi:E\longrightarrow\mathbb{P}^1$. This {\it inflection divisor} is the zero locus of a Wronskian of partial derivatives of a space of regular sections of a line bundle $L$. When $\mu=k-1$ this {\it linear series} is complete, and it is well-known that inflection points of a complete series on an elliptic curve are in bijection with torsion points of order $\deg(L)$; see, e.g., \cite{BCG}.

\subsection{Formal definition} Suppose that $E=(E,O)$ is given by the affine equation $y^2=f(x)$  and  $O=\infty$, so that $R_\pi=E[2]$. Over the open set $E_y=E\setminus R_\pi$, the Wronskian whose zeroes describe the inflection locus of the linear series with basis
\[
\mc{F}=\{1,x,\dots,x^k, y, yx, \dots, yx^{\mu-1}\}
\]
is precisely the determinant of the matrix
\[
M(\mu,k):=(D^j(x^iy))_{\substack{0\leq i\leq \mu-1\\k+1\leq j\leq k+\mu\\}}
\]
where $D=\frac{d}{dx}$. %For example $M(k,1)=D^{k+1}y$, and  $M(k,2) = \left( \begin{smallmatrix} D^{k+1}y&D^{k+2}y\\ D^{k+1}(xy)&D^{k+2}(xy) \end{smallmatrix} \right)$. 
Accordingly, we set 
\[
I(\mu,k):=\text{div}(\det M(\mu,k)).
\]
%thus $I(k,n)$  is supported on $E_y$.  For example, $I(k,1)=\text{div }D^{k+1}y$ and $I(k,2)=\text{div }(D^{k+1}yD^{k+2}(xy)-D^{k+2}yD^{k+1}(xy))$.
\section{Properties of inflection divisors}
\subsection{Basic structure} Each divisor $I(\mu,k)$ is determined by an {\it inflection polynomial} $P_{\mu,k}(x)\in K[x]$ defined by

%\begin{thm} 
%For  $k$ and $n$ fixed we have
%\begin{equation}
    %\label{GDP}
    \[
    \det M(\mu,k)=(f^{-(k+1)}y)^{\mu}P_{\mu,k}(x).
%\end{equation}
\]
%\end{thm}

%Proof. Let $n=1$. The proof is by induction on $k$. First note that $Dy=f^{-1}y\frac{f'}{2}=f^{-1}yP_{0,1}(x)$. Suppose that $\text{det }(M(k,1))=D^{k+1}y=f^{-(k+1)}yP_{k,1}(x)$, then $\text{det }(M(k+1,1))=D^{k+2}y=f^{-(k+2)}yP_{k+1,1}(x)$, where 
%\begin{equation}
%    \label{Recurrence}
%    P_{k+1,1}(x)=D(P_{k,1})f-\frac{2k+1}{2}P_{k,1}D(f).
%\end{equation}

%Let $n>1$. We use the following expression from \cite{BCG}: \begin{equation}\label{equivalence}
%\text{det }(M(k,n))\,=\,\text{det}\biggl(\frac{(k+1+j)!}{(k+1+j-i)!}D^{k+1+j-i}y\biggr)_{ 0\leq i,j\leq n-1},
% \end{equation}
 
% and the result follows.$\blacksquare$
 
%\medskip
%\textbf{Remark.} The expression \eqref{equivalence} is very near of being a traditional Wronskian $\text{Wr}(f_0,\ldots,f_{n-1})=\text{det }(D^if_j)_{ 0\leq i,j\leq n-1}$ of an $n$-tuple of functions  $(f_0,\ldots,f_{n-1})$ in the following sense.

%For $n\geq2$ and $0\leq j\leq n-1$, we define $$f_j=f_j(m,n)=(m+j)\cdots(m+j+2-n)D^{m+j+1-n}y,$$ and for $0\leq i,j\leq n-1$ we define $$c_{i,j}=\frac{1}{(m+j+2-n)\cdots(m+j+2-n+i-1)}.$$

%Then $\text{det }(M(k,n))=\text{det }(c_{i,j}D^if_j)_{ 0\leq i,j\leq n-1}$.
 
 %\medskip
 %\par\noindent
 %So, if $P_{n,k}(x)=\prod_{i}(x-p_i)^{n(p_i)}$ in the algebraic closure $\overline{K}$ of $K$, then  $I(k,n)=\sum_{i}n(p_i)(p_i,\pm\sqrt{f(p_i)})$. 
 
 \medskip
The zeroes of $P_{\mu,k}$ give the {\it $x$-coordinates} of points belonging to the inflection divisor $I(\mu,k)$. Clearly $I(\mu,k)$ is invariant under the involution $(x,y)\mapsto (x,-y)$ on $E_y$, and it follows that $\deg_x(P_{\mu,k})=2\mu(k+1)$. More to the point, if we write $f=x(x-1)(x-\la)$ in Legendre form with $\la \in \mb{C}$, we may view $P_{\mu,k}$ as a function of both $x$ and $\la$, and $\deg_\lambda(P_{\mu,k})=\mu(k+1)$.
 
%From now on we will focus on the polynomials $P_{n,k}(x)$. 
 
 \medskip
 The case $\mu=1$ is distinguished: in that case the matrix $M=M(\mu,k)$ is a $1 \times 1$ matrix, and the corresponding family of inflection polynomials $P_{1,k}$ is described inductively by the rule
 \begin{equation}\label{inflection_recurrence}
 P_{1,k+1}=D(P_{1,k})f+ (-k+1/2)P_{1,k}D(f)
 \end{equation}
 for all $k \geq 0$, subject to the seed datum $P_{1,0}=\frac{1}{2} D(f)$.
 
%\medskip
%\begin{ques}
%The expression \eqref{inflection_recurrence} is reminiscent of commutation relations associated for operators on a Fock space. Is there a deeper reason for this, e.g. an underlying integrable hierarchy?
%\end{ques}

\medskip
Inflection polynomials associated with higher values of $\mu$ may be realized as polynomials in the ``basic" inflection polynomials $P_{1,k}$, as follows.

 \begin{lemma}\label{general_inflection_polys} Given positive integers $\mu\geq2$ and $k \geq 3$, set $n=k+1$. There exists a homogeneous polynomial $Q_{\mu,n}\in\mathbb{Z}[t_{1-\mu},\ldots t_0,\ldots,t_{\mu-1}]$ of degree $\mu$ for which
 \begin{equation}
     P_{\mu,k}=Q_{\mu,n}|_{t_\ell=P_{1,n+\ell-1}}
 \end{equation}
where \[
     Q_{\mu,n}(t_{1-\mu},\ldots t_0,\ldots,t_{\mu-1}):=\det\left((n+j)_{(i)} t_{j-i}\right)_{0\leq i,j\leq \mu-1}
 \]
 and where, for any non-negative integers $a$ and $i$, $a_{(i)}= \frac{a!}{(a-i)!}$ denotes the $i$-th falling factorial of $a$.
 \end{lemma}
 
 \begin{proof}
 We make use of the following expression from \cite{BCG}: \begin{equation}\label{equivalence}
\text{det }(M(\mu,k))\,=\,\text{det}\biggl(\frac{(k+1+j)!}{(k+1+j-i)!}D^{k+1+j-i}y\biggr)_{ 0\leq i,j\leq \mu-1}.
\end{equation}
 
Substituting $t_{j-i}=D^{k+1+j-i}y$ and $n=k+1$ in \eqref{equivalence} yields
%\begin{equation} 
\[
\text{det }(M(\mu,k))\,=\,\text{det}\biggl(\frac{(n+j)!}{(n+j-i)!}t_{j-i}\biggr)_{ 0\leq i,j\leq \mu-1}=Q_{\mu,n}|_{t_\ell=D^{n+\ell}y},
\]
%\end{equation}
 
The desired result now follows from the fact that $D^{n+\ell}y=f^{n+\ell}yP_{1,n+\ell-1}$, since $Q_{\mu,n}|_{t_\ell=D^{n+\ell}y}=(f^{-n}y)^\mu Q_{\mu,n}|_{t_\ell=P_{1,n+\ell-1}}$.
\end{proof}

 The upshot of Lemma~\ref{general_inflection_polys} is that the projective geometry of the inflection divisor $I(\mu,k)$ is controlled by the inductive prescription \eqref{inflection_recurrence} for basic inflection polynomials, and also (in a more obscure way) by the hypersurfaces of degree $\mu$ defined by the $Q_{\mu,n}$ in $\mb{P}^{2(\mu-1)}$. Note that the division polynomials are essentially the inflection polynomials $P_{k-1,k}$.
 
 {\bf \fl Remarks.}
\begin{itemize}
\item[i.] 
When $m\geq3$, a basis for the complete linear series on $E$ induced by the divisor $m\cdot\infty$ is $\mathcal{F}=\{x^i,x^jy\}_{i,j}$, where 
\begin{equation*}
    \begin{cases}
0\leq i\leq m/2,\:0\leq j\leq (m-4)/2&\text{ if }m\text{ is even},\\
0\leq i\leq (m-1)/2,\:0\leq j\leq (m-3)/2&\text{ if }m\text{ is odd}.\\
\end{cases}
\end{equation*}
It should be straightforward to adapt our degeneration-based analysis of inflection points to account for the odd case  as well.
\item[ii.] Whenever the curve $E$ is defined over a field of positive characteristic $p\neq 2,3$, it has a Weierstrass equation and the associated Wronskian $\det M(\mu,k)$ may still be defined by using Hasse derivatives in place of the differential operators $D^i=\frac{d^i}{dx^i}$. St\"ohr and Voloch have used these Wronskians to carry out (refinements of) rational point counts for algebraic curves defined over finite fields \cite{SV}.% associated to complete linear series defined on them (see \cite{SV}).
\item[iii.] The Wronskian of a rank-$r$ linear series $(L,V)$ on an arbitrary algebraic curve $C$ naturally defines an {\it Euler class}, associated to (the determinant of) the jet bundle $J^{r+1}(L)$ on $C$ whose fiber in a point $p$ is $H^0(L/L(-(r+1)p))$. Moreover, when $C$ is (hyper)elliptic and $L$ arises via pullback from $\mb{P}^1$, the jet bundle in question is {\it (relatively) orientable}, which ensures that the Wronskian defines an {\it arithmetic} Euler class in the sense of $\mb{A}^1$-homotopy theory \cite{KW1,KW2}. In \cite{CDH} we calculate a global arithmetic inflection formula over an arbitrary field of characteristic $p \neq 2,3$, and we analyze the geometric meaning of its constituent local arithmetic inflection indices over $\mb{F}_q$.
%The zero locus of a Wronskian naturally defines an {\it Euler class}, so is potentially amenable to the methods of $\mb{A}^1$-homotopy theory as developed in \cite{KW1,KW2}. It would be interesting to see to what extent our results may be recovered and generalized to arbitrary fields within that framework. 
\end{itemize}
 
\subsection{Symmetries of inflection polynomials}
The inductive formula \eqref{inflection_recurrence} has a number of interesting consequences.

\begin{lemma}\label{key_poly_symmetry}
For every positive integer $k \geq 1$, we have
\begin{equation}\label{symmetry_property}
P_{1,k}(x,\la)= P_{1,k}(x,z) \text{ and }P_{1,k}(x+1,\la+1)= P_{1,k}(-x,-\la).
\end{equation}
Here by $P_{1,k}(x,z)$ we mean the polynomial obtained by homogenizing with respect to $z$ to obtain a degree-$2(k+1)$ polynomial in $\mb{Q}[x,\la,z]$; and then dehomogenizing with respect to $\la$ to obtain a degree-$2(k+1)$ polynomial in $\mb{Q}[x,z]$.
\end{lemma}

\begin{proof}
See \cite[Lemma 4.2]{CG}.
\end{proof}

Lemma~\ref{key_poly_symmetry} implies that the monodromy group associated to the projection from the point $[0:0:1]$ of the projective closure $\overline{\mathcal{C}(1,k)}$ of the {\it inflectionary curve} $\mathcal{C}(1,k)$ defined by $P_{1,k}=0$ inside $\mb{P}^2_{x,\lambda,z}$ contains transpositions that freely permute the points $p_1$, $p_2$, and $p_3$ with coordinates $[0:0:1]$, $[0:1:0]$, and $[1:1:1]$. In particular, the singularities of $\overline{\mathcal{C}(1,k)}$ in these three points are analytically isomorphic.

\begin{conj}\label{singularity_conj}
For every positive integer $k \geq 1$, the plane curve $\overline{\mathcal{C}(1,k)}$ is nonsingular along $\mb{P}^2 \setminus \{p_1,p_2,p_3\}$.
\end{conj}

{\bf \fl Remarks.}
\begin{itemize}
\item[iv.] The second symmetry in \eqref{symmetry_property} suggests that modular properties of $E$ and $\la$ are at play in Conjecture~\ref{singularity_conj}.
\item[v.] The canonical (i.e. N\'eron--Tate) height of a torsion point on $E$ is zero. In \cite{S2}, Silverman obtains bounds on the canonical heights of inflection points for pluricanonical series on hyperelliptic curves. It would be interesting to extend his analysis to our incomplete series on the elliptic curve $E$, as doing so would quantify the extent to which inflection points ``stray" from the torsion lattice on the universal cover of $E$. %It would be interesting to obtain bounds on the canonical heights of inflection points, as these would quantify the extent to which inflection points ``stray" from the torsion lattice on the universal cover of $E$.
\end{itemize}

\subsection{Genera of inflectionary curves}
%We should be able to prove the following by induction.
The inductive prescription~\eqref{inflection_recurrence} for the inflection polynomials $P_{1,k}$ strongly suggests the following is true.
\begin{conj}\label{singularity_conj_one}
For $k\geq1$, let $\Delta_1(k)=\text{Conv}\{(0,k+1),(k-1,k+1),(k-1,2),(2k-2,2)\}$ and $\Delta_2(k)=\text{Conv}\{(2k,1),(2k+1,1),(2k+1,0),(2k+2,0)\}$. Then
\begin{enumerate}
    \item $P_{1,k}$ has support:\[
    \text{Supp}(P_{1,k})=(\Delta_1(k)\cap\mathbb{Z}^2)\cup(\Delta_2(k)\cap\mathbb{Z}^2)
\]
\item Let $\sigma_k:\text{Supp}(P_{1,k})\xrightarrow{} \text{Supp}(P_{1,k})$ be the reflection along the diagonal $\text{Conv}\{(0,k+1),(2k+2,0)\}$, this is  $\sigma_k(i,j)=(i,2k+2-i-j)$. If $a_{(i,j)}^{(k)}x^i\lambda^j$ is a monomial of $P_{1,k}$, then $a_{(i,j)}^{(k)}=a_{\sigma_k(i,j)}^{(k)}$.
\end{enumerate}
\end{conj}

Conjecture~\ref{singularity_conj_one} predicts that whenever $k\geq2$, the lower faces of the Newton polygon of $P_{1,k}$ that contribute to the singularity of the inflectionary curve $\overline{\mc{C}(1,k)}$ in $(0,0)$ are $\Gamma_1(k):=\text{Conv}\{(0,k+1),(k-1,2)\}$ and $\Gamma_2(k)=\text{Conv}\{(k-1,2),(2k+1,0)\}$. In particular, we have 
\[
P_{1,k}|_{\Gamma_1(k)}=y^2Q_{k-1}(x,y),\quad P_{1,k}|_{\Gamma_2(k)}=x^{k-1}(a_ky^2+b_kx^{k+2})
\]
where $Q_{k-1}(x,y)=\prod_{j=1}^{k-1}(a_j x+b_j y)$.
% is a  quasi-homogeneous polynomial of type $(1,1)$ and degree $k-1$. 
Geometrically, this means that the singularity in $(0,0)$ is the union of $k-1$ smooth branches corresponding to the linear factors of $Q_{k-1}$, together with an additional singularity of analytic type $y^2+x^{k+2}=0$. Assuming all of these intersect transversely, we would expect the following to hold. 

\begin{conj}\label{singularity_conj_two}
The delta-invariant of the singularity of the inflectionary curve $P_{1,k}=0$ in $(0,0)$ is %$\delta=\lfloor \frac{k^2+1}{2} \rfloor-1$.
$\delta=\lfloor \frac{k^2}{2} \rfloor+ k$.
\end{conj}

Indeed, one arrives at Conjecture~\ref{singularity_conj_two} by viewing $\delta$ as a ``local number of nodes". The $(k-1)$-fold point defined by $Q_{k-1}$ generically will have delta-invariant $\binom{k-1}{2}$, while $y^2+x^{k+2}=0$ has delta-invariant $\lfloor \frac{k}{2}\rfloor +1$. Finally, their union will generically have delta-invariant equal to $2(k-1)$ plus their sum.

Conjectures \ref{singularity_conj} and \ref{singularity_conj_two}, in tandem with the degree-arithmetic genus formula for plane curves, predict that the geometric genus of $\overline{\mc{C}(1,k)}$ is $p_g= \binom{2k+1}{2}- 3\lfloor \frac{k^2}{2} \rfloor-3k$.

\section{Reality phenomena}
\subsection{Separability}
In the papers \cite{BCG} and \cite{CG}, our primary objective was constructing real linear series with many real inflection points on (hyper)elliptic curves $X$. The situation when $X=E$ is elliptic and the real locus $E(\mb{R})$ has two connected components is distinguished.

\begin{conj}\label{real_separability_conjecture} When the polynomial $f=x(x-1)(x-\la)$ has three distinct real roots (i.e., when $\la \in \mb{R}$), each corresponding inflection polynomial $P_{\mu,k}$ has only simple roots in $x$ away from $\{0,1\}$ for every fixed value of $\la$.
%on the complement of $\{0,1\}$, for all $\mu, k \geq 1$ with $\mu<k$.
\end{conj}
In particular, Conjecture~\ref{real_separability_conjecture} predicts that each real zero $x=\ga$ of $P_{\mu,k}$ lifts to either 2 or 0 real points of $I(\mu,k)$, depending upon whether the number $P_{\mu,k}(\ga,\la)$ is positive or not. 

\subsection{Real loci of inflectionary curves in the maximally-real case}

\begin{conj}\label{numerology_of_key_polynomial_roots}
Let $\mu \geq 1$ and $k \geq \mu+1$ be nonnegative integers. %, and let $X_{\mu,n} \sub \mb{R}^2_{x,\la}$ denote the real affine curve defined by the implicit equation $P_{\mu,n}(x,\la)=0$. Then 
Assume that $f$ is associated with a real Legendre parameter $\la$. For every fixed value of $\la\neq0,1$, the corresponding inflection polynomial $P_{\mu,k}$ has precisely either $\mu$ or $2\mu$ real roots $x=\gamma$ such that $f(\gamma,\la) > 0$, depending upon whether $k-\mu$ is even or odd.
\end{conj}

 %The real locus of the inflectionary curve $\overline{\mathcal{C}(\mu,k)}$ defined by $P_{\mu,k}=0$ seemingly always has two connected components (just like the real locus of $E$ itself) and its
The most striking evidence in favor of Conjecture~\ref{numerology_of_key_polynomial_roots} is graphical in nature; see \cite{CG} and Figure 1 above.
The topology of the real locus of the inflectionary curve $\overline{\mathcal{C}(\mu,k)}$ defined by $P_{\mu,k}=0$ seemingly is controlled by the singularities in the distinguished points $[0:0:1]$, $[0:1:0]$, and $[1:1:1]$.

\begin{figure}\label{graphical_evidence}
    \centering
    \includegraphics{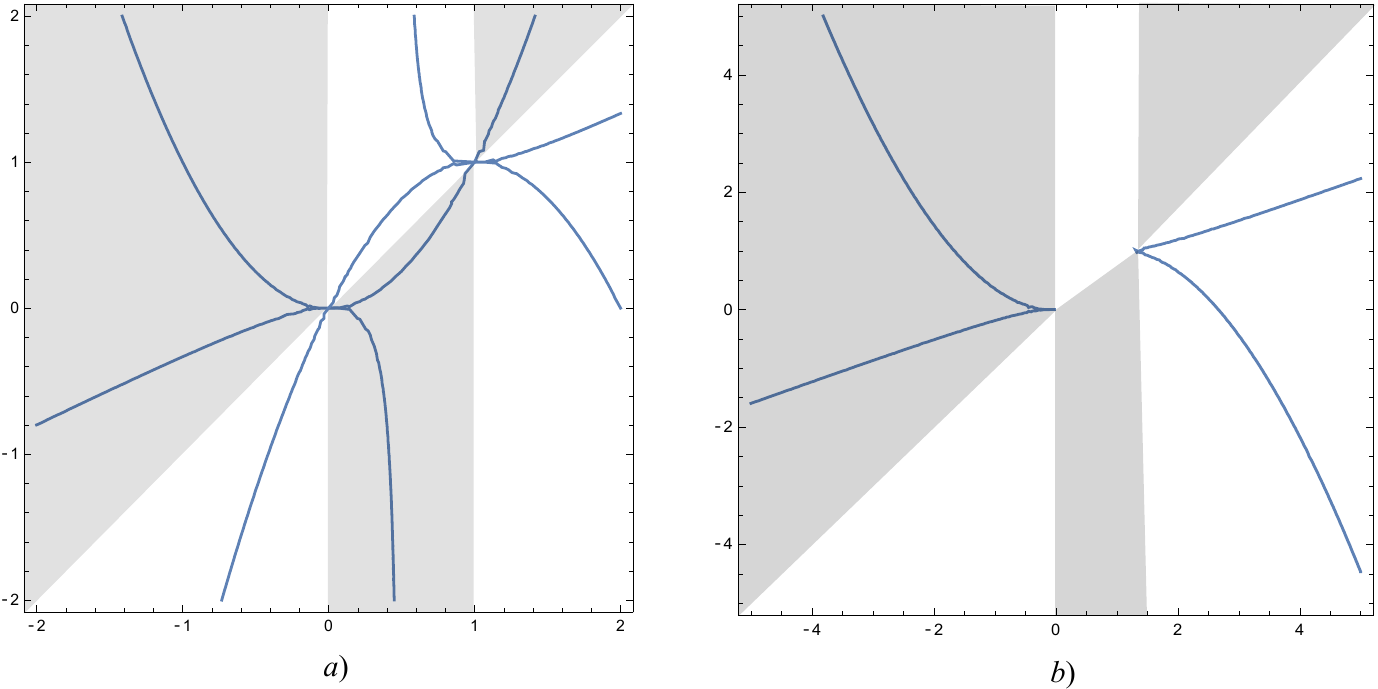}
    \caption{Real loci of the inflectionary curves $\mc{C}(1,2)$ and $\mc{C}(1,3)$, respectively. The regions where $f>0$ are shaded.}
    \label{fig:my_label}
\end{figure}

\end{document}